\documentclass[reqno,a4paper]{article}

\usepackage{ifpdf}
\usepackage{graphicx}
\usepackage{epsfig}
\usepackage{color}

\ifpdf
\usepackage[pdftex=true,a4paper]{hyperref}
\hypersetup{
colorlinks=false,
breaklinks=true,
pdfstartview={FitH},
unicode=true
}
\else
\usepackage[a4paper,hypertex]{hyperref}
\fi

\usepackage{a4wide}
\usepackage[utf8]{inputenc}
\usepackage[english,ngerman]{babel}
\usepackage{amssymb}
\usepackage{amsmath}
\usepackage{amsthm}
\usepackage{mathrsfs}
\usepackage{mathptmx}
\usepackage{bbold} 
\usepackage[scaled=.90]{helvet}
\usepackage{times}
\usepackage{savesym}
\usepackage{todonotes}
\usepackage{extarrows}

\numberwithin{equation}{section}

\newcommand{\titlestring}{High-dimensional Gaussian 
fields with isotropic increments seen through spin glasses}
\newcommand{\authorstring}{Anton Klimovsky}
\newcommand{\subjectstring}{Primary: 60K35; Secondary: 82B44, 82D30, 60G15,
60G60, 60F10}

\newcommand{\keywordsstring}{Gaussian random fields, isotropic increments,
random energy model, hierarchical replica symmetry breaking, Parisi Ansatz}

\ifpdf
\hypersetup{
pdftitle = {\titlestring},
pdfauthor = {\authorstring},
pdfsubject = {\subjectstring},
pdfkeywords = {\keywordsstring},
pdfcreator = {},
pdfproducer = {}
}
\fi

\renewcommand{\P}{\mathbb {P}}
\newcommand{\E}{\mathbb {E}}

\newcommand{\R}{\mathbb {R}}

\newcommand{\N}{\mathbb {N}}

\newcommand{\I}{\mathbb{1}} 
\newcommand{\dd}{{\rm d}}

\DeclareMathOperator{\var}{Var}
\DeclareMathOperator{\cov}{Cov}

\newcommand{\eps}{\varepsilon}

\newtheorem{theorem}{Theorem}[section]

\newtheorem{corollary}{Corollary}[section]

\newtheorem{lemma}{Lemma}[section]

\newtheorem{remark}{Remark}[section] 

\ifx\hypersetup\undefined\else
\fi

\def\RPC{\mathrm{RPC}}

\allowdisplaybreaks

\begin{document}

\selectlanguage{english}

\begin{center}
\LARGE{\titlestring}
\end{center}
\vskip1cm
{\large Anton~Klimovsky}\footnote{
Research supported in part by the European
Commission (Marie Curie fellowship, project PIEF-GA-2009-251200); bilateral DFG-NWO
Forschergruppe 498 and the Hausdorff Research Center for Mathematics (Junior Trimester
Program on Stochastics 2010).
}
\\
\emph{EURANDOM, Eindhoven University of Technology, The Netherlands}
\\
e-mail: \texttt{a.klymovskiy@tue.nl}
\vskip0.5cm
\begin{center}
\today
\end{center}
\vskip0.5cm \noindent \textbf{AMS 2000 Subject Classification:} \subjectstring.
\vskip0.5cm
\begin{center}
\textbf{Abstract}
\end{center}

\noindent We study the free energy of a particle in (arbitrary) high-dimensional
Gaussian random potentials with isotropic increments. We prove a computable
saddle-point variational representation in terms of a Parisi-type functional
for the free energy in the infinite-dimensional limit. The proofs are based on
the techniques developed in the course of the rigorous analysis of the
Sherrington-Kirkpatrick model with vector spins.

\vskip0.5cm \noindent

\section{Introduction}

Recently, considerable (renewed) attention in the theoretical physics literature
has been devoted to Gaussian random fields with isotropic increments viewed as
random potentials, see, e.g, the works by Fyodorov and Sommers
\cite{FyodorovSommers2007}, Fyodorov and Bouchaud \cite{FyodorovBouchaud2008},
and references therein. In particular, it was heuristically argued in these
works that Parisi's theory of hierarchical replica symmetry breaking (Parisi
Ansatz, cf.~\cite{MezardParisiVirasoro1987}) is applicable in this context. In
the probabilistic context, these results provide rather sharp information about
the extremes of the strongly correlated fields with high-dimensional correlation
structures, which is a challenging area of probability
theory~\cite{TalagrandSpinGlassesBook2003,BovierBook2006,BolthausenBovierBook2007, BoutetdeMonvelBovierBook2009,TalagrandSpinGlassesBookVolOneSecEd2011, TalagrandSpinGlassesBookVolTwoSecEd2011}.

In this note, we initiate the rigorous derivation of the results of
\cite{FyodorovSommers2007,FyodorovBouchaud2008}. We concentrate on the
computation of the \emph{free energy} of a particle subjected to arbitrary
high-dimensional \emph{Gaussian random potentials with isotropic increments}. 
In the high-dimensional limit, we derive a computable \emph{saddle-point
representation} for the free energy,  which is similar to the Parisi formula for
the Sherrington-Kirkpatrick (SK) model of a mean-field spin glass. Our proofs
are based on the \emph{local comparison} arguments for Gaussian fields with
non-constant variance developed in \cite{BovierKlimovskyAS2Guerra2008}, which
are, in turn, based on the ideas of Guerra~\cite{Guerra2003a}, Guerra and
Toninelli~\cite{Guerra-Toninelli-Generalized-SK-2003}, 
Talagrand~\cite{TalagrandParisiFormula2006} and
Panchenko~\cite{panchenko-free-energy-generalized-sk-2005}.

This note is organised as follows. We state our results in
Section~\ref{sec:main-results}. The proofs are given in
Sections~\ref{sec:existence-free-energy} and
\ref{sec:proof-saddle-point-product-case}. In Section~\ref{sec:outlook}, we give
an outlook and announce some important consequences of the results of this note.
In the Appendix, we provide some complementary information for the reader's
convenience.

\section{Setup and main results}
\label{sec:main-results}
Consider the \emph{Gaussian random field} with \emph{isotropic increments} $ X =
X_N = \{ X_N(u) : u \in \R^N \}$, $N \in \N$. The adjective ``isotropic'' means
here that the law of the increments of the field $X$ is invariant under
\emph{rigid motions} ($=$ translations and rotations) in $\R^N$. We are
interested in the case $N \gg 1$ and in the case of \emph{strongly correlated
fields with high-dimensional correlation structure}. Therefore, we assume that
the field $X_N$ satisfies
\begin{align}
\label{eq:incremental-covariance}
\E
\left[
(X_N(u) - X_N(v))^2
\right]
=
D
\left(
\frac{1}{N}\Vert u - v \Vert^2_2
\right)
=:
D_N(
\Vert u - v \Vert^2_2
)
,
\quad
u,v \in \R^N
,
\end{align}
where $\Vert \cdot \Vert_2$ denotes the \emph{Euclidean norm} on $\R^N$ and
the \emph{correlator} $D: \R_+ \to \R_+$ is any admissible function. Complete
characterisation of all correlators $D$ that are admissible in
\eqref{eq:incremental-covariance}, for all $N$, is known, see
Theorem~\ref{thm:characterisation-correlation-functions}.  Note that the law of
the field $X_N$ is determined by \eqref{eq:incremental-covariance} only up to an
additive shift by a Gaussian random variable. In what follows, without loss of
generality, we assume that $X_N(0)=0$.

We are interested in the asymptotic behaviour of the \emph{extremes} of the
random field $X_N$ on the sequence of the \emph{particle state spaces} $S_N
\subset \R^N$ as $N \uparrow +\infty$. The state spaces are assumed to be equipped
with a sequence of \emph{a priori reference measures} $\{ \mu_N \} \subset
\mathcal{M}_{\mathrm{finite}}(S_N)$. We now define the main quantities of
interest in this work. Consider the \emph{partition function}
\begin{align}
\label{eq:partition-function}
Z_N(\beta)
:=
\int_{S_N} \mu_N(\dd u) \exp \left( \beta \sqrt{N} X_N(u) \right)
,
\quad
\beta \in \R
.
\end{align}
 We view \eqref{eq:partition-function} as an exponential functional of the field
$X_N$, which is parametrised by the \emph{inverse temperature} $\beta$.
Heuristically, for large $\beta$ (i.e., $\beta \uparrow +\infty$), the maxima of
the field $X_N$ give substantial contribution to the integral
\eqref{eq:partition-function}. The $N$-scalings in
\eqref{eq:partition-function}, \eqref{eq:incremental-covariance} and the
``size'' of $S_N$ are tailored for studying the large-$N$ limit of the
\emph{quenched log-partition function}:
\begin{align}
\label{eq:log-partition-function}
p_N(\beta)
:=
\frac{1}{N}\log Z_N(\beta)
,
\quad
\beta \in \R
.
\end{align}
 For comparison with the theoretical physics literature, let us note that there one
conventionally substitutes $\beta \mapsto -\beta$ in
\eqref{eq:partition-function} (this has no effect on the distribution of $Z_N$
due to the symmetry of the centred Gaussian distribution of the field $X_N$),
and considers instead of \eqref{eq:log-partition-function} the \emph{free energy}
\begin{align}
\label{eq:free-energy}
f_N(\beta)
:=
- \frac{1}{\beta} p_N(\beta)
,
\quad
\beta \in \R_+
.
\end{align}

\textbf{Assumptions.}
Informally, we require the particle state space $S_N$ to have an
exponentially growing in $N$ volume (respectively, cardinality, if $S_N$ is
discrete). In particular, using physics parlance, this assures that the
\emph{entropy} competes with the \emph{energy} (given by the random field $X_N$)
on the same scale. More formally, we assume
\begin{equation}
\label{eq:product-space}
S_N := S^N
,
\quad
S \subset \R
.
\end{equation}
Let
$\mu \in \mathcal{M}_{\text{finite}}(S)$ be such that
the origin is contained in the interior of the convex hull of the support of
$\mu$. Define $\mu_N := \mu^{\otimes N}
\in
\mathcal{M}_{\text{finite}}(S_N)$. A canonical example is the \emph{discrete
hypercube} $S_N := \{-1; 1\}^N$ equipped with the uniform a priori measure,
i.e., $\mu(\{u\}) := 2^{-N}$, for all $u \in S_N$.

\textbf{Parisi-type functional.}
To formulate our results on the limiting log-partition function, we need the
following definitions. Given $r \in \R_+$, consider the space of the
\emph{functional order parameters}
\begin{align}
\label{eq:chap-1:space-of-order-parameters}
\mathcal{X}(r)
:=
\{
x: [0;r] \to [0;1]
\mid
&
\text{
$x$ is non-decreasing càdlàg,
}
x(0)=0,
x(r)=1
\}
,
\end{align}
It is convenient to work with the space of the \emph{discrete order
parameters}
\begin{align}
\label{eq:chap-1:space-of-discrete-order-parameters}
\mathcal{X}^\prime_n(r)
:=
\{
x \in \mathcal{X}(r)
\mid
\text{$x$ is piece-wise constant with at most $n$ jumps}
\}
.
\end{align}
Let us denote the \emph{effective size} of the particle state space by
\begin{align}
\label{eq:state-space-effective-size}
d := \sup_N
\left(
\frac{1}{N}
\sup_{u \in S_N}
\Vert u \Vert_2^2
\right)
.
\end{align}
For what follows, it is enough to assume that $r \in [0;d]$ in
\eqref{eq:chap-1:space-of-order-parameters}. Note that, in case
\eqref{eq:product-space}, $d = \sup_{u \in S} u^2$.

Now, let us define the non-linear functional that appears in
the variational formula of our main result. We do it in three
steps:
\begin{enumerate}
\item Given large enough $M \in \R_+$, define the regularised derivative
$D^{\prime,M}: \R_+ \to \R$ of the correlator $D$ as
\begin{equation}
\label{eq:regularized-derivative}
D^{\prime,M}(r)
:=
\begin{cases}
D^{\prime}(r)
,
&
r \in [1/M;+\infty)
,
\\
M
,
&
r \in [0;1/M)
.
\end{cases}
\end{equation}
Given $r, M \in \R_+$, define the function $\theta_r^{(M)} : [-r;r] \to \R$ as
\begin{align}
\label{eq:theta-def}
\theta_r^{(M)}(q)
:=
q D^{\prime,M}(2(r-q))+\frac{1}{2}D(2(r-q))
,
\quad
q \in [-r;r]
.
\end{align}

\item
Given $r \in \R_+$,  $x
\in
\mathcal{X}(r)$ and
the (regular enough) \emph{boundary condition} $h: \R \to \R$, consider the semi-linear parabolic \emph{Parisi's terminal
value problem}:
\begin{align}
\label{eq:parisi-pde}
\begin{cases}
\partial_q f(y,q)
+
\frac{1}{2}
D^{\prime,M}(2(r-q))
\left(
\partial^2_{qq} f(y,q)
+
x(q)
\left(
\partial_{y} f(y,q)
\right)^2
\right)
=
0
,
&
(y,q)
\in
\R \times (0,r)
,
\\
f(y,1)
=
h(y)
,
&
y \in \R
.
\end{cases}
\end{align}
Let $f_{r,x,h}^{(M)}:[0;1] \times \R_+ \to \R$ be the unique solution of
\eqref{eq:parisi-pde}. Solubility of the Parisi terminal value problem \eqref{eq:parisi-pde}, its relation to the Hamilton-Jacobi-Bellman equations and stochastic control problems  is discussed in a more general multidimensional context
in~\cite[Section~6]{BovierKlimovskyAS2Guerra2008}.

\item Given the family of the (regular
enough for \eqref{eq:parisi-pde} to be solvable) \emph{boundary conditions}
\begin{equation}
\label{eq:boundary-conditions}
g :=
\{g_\lambda: \R \to \R \mid \lambda \in \R \}
,
\end{equation}
and given $r \in [0;d]$, define the 
\emph{local Parisi functional}
$
\mathcal{P}(\beta,r,g): \mathcal{X}(r) \to \R
$
as
\begin{align}
\label{eq:local-parisi-functional}
\mathcal{P}(\beta,r,g)[x]
:=
\lim_{M \uparrow +\infty}
\left(
\inf_{\lambda \in \R}
\left[
f_{r,x,g_\lambda}^{(M)}(0,0)
-\lambda r
\right]
-
\frac{\beta^2}{2}
\int_0^1
x(q)
\dd \theta_r^{(M)}(q)
\right)
,
\quad
x \in \mathcal{X}(r)
.
\end{align}
In \eqref{eq:local-parisi-functional}, the integral with respect to $\theta_r^{(M)}$ is
understood in the Lebesgue-Stiltjes sense.
\end{enumerate}

\textbf{Main results.} Let us start by recording the basic convergence result
for the log-partition function.
\begin{theorem}[Existence of the limiting free energy]
\label{thm:existence-of-limiting-free-energy}
For any $\beta>0$, the large $N$-limit of the log-partition function exists and
is a.s. deterministic:
\begin{align}
\label{eq:free-energy-convergence}
p_N(\beta)
\xlongrightarrow[N \uparrow +\infty]{}
p(\beta)
,
\quad
\text{almost surely and in $L^1$.}
\end{align}
In addition, for any $N \in \N$, the following concentration of measure inequality holds
\begin{equation}
\label{eq:concentratino-inequality}
\P
\left\{
|
p_N(\beta)
-
\E\left[
p_N(\beta)
\right]
|
>
t
\right\}
\leq
2 \exp
\left(
- \frac{
N t^2
}{
4 D(d)
}
\right)
,
\quad
t \in \R_+
.
\end{equation}
\end{theorem}
The main result of this work is the following variational representation for the
limiting log-partition function in terms of the Parisi
functional~\eqref{eq:local-parisi-functional}.
\begin{theorem}[Free energy variational representation, comparison with cascades]
\label{thm:free-energy}
Assume \eqref{eq:product-space}. Let the family of boundary conditions \eqref{eq:boundary-conditions} be defined as
\begin{align}
\label{eq:product-space-terminal-condition}
g_{\lambda}(y)
:=
\log
\int_S
\mu(\dd u)
\exp
\left(
\beta u y
+
\lambda u^2
\right)
,
\quad
y \in \R
.
\end{align}

Then, for all $\beta \in \R$,
\begin{align}
\label{eq:free-energy-variational-plus-remainder-term}
p(\beta)
:=
\sup_{r \in [0;d]} \inf_{x \in \mathcal{X}(r)}
\left(
\mathcal{P}(\beta, r, g)[x]
-
\mathcal{R}(r)[x]
\right)
,
\quad
\text{almost surely and in $L^1$,}
\end{align}
where the \emph{remainder term} $\mathcal{R}(r): \mathcal{X}(r) \to \R_+$ is a
functional on $\mathcal{X}(r)$ taking non-negative values (see
\eqref{eq:remainder:local-limiting-guerra-remainder} for the definition). 
\end{theorem}
The sign-definiteness of the remainder term $\mathcal{R}(r)$ immediately implies the following
bound.
\begin{corollary}[Log-partition function upper bound]
\label{thm:free-energy-upper-bound}
For all $\beta \in \R$,
\begin{align}
\label{eq:free-energy-upper-bound}
p(\beta)
\leq
\sup_{r \in [0;d]} \inf_{x \in \mathcal{X}(r)}
\mathcal{P}(\beta, r, g)[x]
,
\quad
\text{almost surely.}
\end{align}
\end{corollary}
\begin{remark}
\label{rem:correlator-singularity} In the
case~\eqref{eq:long-range-incremental-covariance-representation}, the field
\eqref{eq:spherically-restricted-correlator} has a feature, which is not within
the assumptions typically found in the literature
\cite{Guerra2003a,Guerra-Toninelli-Generalized-SK-2003,TalagrandParisiFormula2006,TalagrandSphericalSK,panchenko-free-energy-generalized-sk-2005}: the correlator $D$ is not of class $C^1$, namely, $D$ can have a singular derivative at $0$. To deal with the singularity, we need a regularisation procedure, cf.~\eqref{eq:regularized-derivative} and \eqref{eq:local-parisi-functional}.
\end{remark}

\textbf{Heuristics.} It is natural to ask the following questions: Why is Parisi's theory of
hierarchical replica symmetry breaking \cite{MezardParisiVirasoro1987} (which is
usually behind the functionals of the type~\eqref{eq:local-parisi-functional}) applicable to
Gaussian fields with isotropic increments satisfying
\eqref{eq:incremental-covariance}?  Where are the ``interacting spins'' in the
present context?

A hint is given by the following observation. Define 
\begin{equation}
\label{eq:overlap}
\langle
u,v 
\rangle_N
:=
\frac{1}{N}
\sum_{i=1}^N
u_i
v_i
,
\quad
u,v \in \R^N
.
\end{equation}
Let us fix $r \in [0;d]$.
By~\eqref{eq:long-range-covariance}, the restriction of the field $X_N$ with
isotropic increments to a sphere with radius $r$ centred at the origin, leads to
the \emph{mixed $p$-spin spherical SK model} (cf.~\cite{TalagrandSphericalSK})
with the following covariance structure
\begin{equation}
\label{eq:spherically-restricted-correlator}
\E
\left[
X_N(u) X_N(v)
\right]
=
D(r)
-
\frac{1}{2}
D(2 (r - \langle u, v \rangle_N))
=:
G_r(\langle u, v \rangle_N)
,
\quad
\Vert u \Vert_2^2 = \Vert v \Vert_2^2 = r N
, 
\end{equation}
where $G_r: \R_+ \to \R$ is given by
\begin{equation}
\label{eq:effective-sk-correlator}
G_r(q) := D(r)-\frac{1}{2}D(2(r-q))
,
\quad 
q \in \R_+
.
\end{equation}
Thus, \eqref{eq:spherically-restricted-correlator} implies that, given $r$, each
field of the type \eqref{eq:incremental-covariance} induces a mixed $p$-spin
spherical SK model with the convex correlation function $G_r$ (see
Remark~\ref{rem:correlator-concavity}). It is this convexity that leads to the
sign-definiteness of the remainder term in
\eqref{eq:as2:guerras-interpolation-speed} and allows for the proof (along the
lines of \cite{TalagrandParisiFormula2006}) of
Theorems~\ref{thm:parisi-type-formula} and \ref{thm:fyodorov-sommers-formula}
for all admissible correlators. 

Our proof of Theorem~\ref{thm:free-energy} exploits the observation
\eqref{eq:spherically-restricted-correlator} and combines it with the
localisation technique of \cite{BovierKlimovskyAS2Guerra2008}. By means of the
large deviations principle, this technique reduces the analysis of the full
log-partition function~\eqref{eq:log-partition-function} to the \emph{local}
one, where \eqref{eq:spherically-restricted-correlator} approximately holds true
everywhere. The price to pay for this reduction is the saddle point variational
principle \eqref{eq:free-energy-variational-plus-remainder-term}, which involves
the Lagrange multipliers that enforce the localisation.

\section{Existence of the limiting free energy}
\label{sec:existence-free-energy}

In this section, we prove Theorem~\ref{thm:existence-of-limiting-free-energy}.

\begin{proof}[Proof of Theorem~\ref{thm:existence-of-limiting-free-energy}]

Proof of \eqref{eq:concentratino-inequality}.
By Remark~\ref{rem:short-long-range-cases}, we have
\begin{align}
\label{eq:variance-rough-upper-bound}
\var \left[
X_N(u)
\right]
=
D_N(\Vert u \Vert_2^2)
\leq
D(d)
,
\quad
u \in \R^N
.
\end{align}
Therefore, the concentration of measure inequality
\eqref{eq:concentratino-inequality} follows from
\cite[Proposition~2.2]{BovierKlimovskyAS2Guerra2008}.

Proof of the convergence \eqref{eq:free-energy-convergence}. The result
can be proved along the lines
of~\cite[Theorem~1]{Guerra-Toninelli-Generalized-SK-2003}. In  \cite[eq.
(7)]{Guerra-Toninelli-Generalized-SK-2003}, it is assumed that the covariance
structure of the random potential depends on the scalar product (overlap) of the
particle configurations in a smooth way. Therefore, using the terminology of
Remark~\ref{rem:short-long-range-cases}, only the short-range case is covered
by~\cite[Theorem~1]{Guerra-Toninelli-Generalized-SK-2003}. Indeed, in that case,
the covariance of the field $X_N$
satisfies~\eqref{eq:stationary-correlation-structure}, where the function $B$ is
analytic and convex, which follows from the
representation~\eqref{eq:correlation-function-representation}. Therefore,
\cite[Theorem~1]{Guerra-Toninelli-Generalized-SK-2003} is applicable with
$Q_N(u,v) := N^{-1} \Vert u-v \Vert_2^2$, for $u,v \in \R^N$.

In the long-range case \eqref{eq:long-range-covariance}, the proof of the
\cite{Guerra-Toninelli-Generalized-SK-2003} requires some care, because the
covariance structure of the field $X_N$ (cf.~\eqref{eq:long-range-covariance})
does not depend on the scalar product \eqref{eq:overlap} only, and, moreover,
the correlator $D$ is not of class $C^1$
(cf.~Remark~\ref{rem:correlator-singularity}). For the reader's convenience, we
now retrace the main parts of this argument. Given $N \in \N$, we prove the
convergence of \eqref{eq:free-energy-convergence} along the subsequences $\{ N_K
:= N^K \}_{K \in \N}$. Convergence along other subsequences then readily
follows. Consider $N$ independent copies $\{ X^{(k)}_{N_{K-1}} \mid k \in [N]
\}$ of the field $X_{N_{K-1}}$. Given an interval $V \subset [0;d]$, define the
\emph{localised state space} as
\begin{equation}
\label{eq:restricted-state-space}
S_N(V) := 
\left\{
u \in S_N
:
\Vert u \Vert_2^2 \in N \cdot V
\right\}
.
\end{equation}
Given a random field $C = \{C_N(u) \mid u \in \R^N \}$, denote the corresponding
\emph{local partition function} by
\begin{equation}
\label{eq:restricted-partition-function}
Z_N(\beta, V)[C]
:=
\int_{S_N(V)} \mu_N(\dd u) \exp \left( \beta \sqrt{N} C_N(u) 
\right)
.
\end{equation}
In what follows, for $u \in \R^N$, $v \in \R^M$,  we denote by $u
\shortparallel v$ the vector in $\R^{N+M}$ obtained by concatenation of $u$
and $v$.
Define the Gaussian field $Y$ as
\begin{equation}
\label{eq:existence-comparison-field}
Y_{N,K}(u^{(1)} \shortparallel u^{(2)} \shortparallel \ldots \shortparallel
u^{(N)}) :=
\frac{1}{\sqrt{N}}
\sum_{k=1}^N
X^{(k)}_{N_{K-1}}(u)
,
\quad
u^{(k)} \in \R^{N_{K-1}}
,
\quad
k \in [N]
.
\end{equation}
Due to independence,
\begin{equation}
\label{eq:}
\begin{aligned}
&
\cov
\left[
Y_{N,K}(u^{(1)} \shortparallel u^{(2)} \shortparallel \ldots \shortparallel
u^{(N)})
,
Y_{N,K}(v^{(1)} \shortparallel v^{(2)} \shortparallel \ldots \shortparallel
v^{(N)})
\right]
\\
&
=
\sum_{k=1}^N
\cov
\left[
X^{(k)}_{N_{K-1}}(u^{(k)})
,
X^{(k)}_{N_{K-1}}(v^{(k)})
\right]
,
\quad
u^{(k)}, v^{(k)} \in \R^{N_{K-1}}
,
\quad
k \in [N]
.
\end{aligned}
\end{equation}
Let us define
\begin{equation}
\label{eq:restricted-partition-function-tilde}
\widetilde{Z}_{N_K}(\beta, V)[C]
:=
\int_{\widetilde{S}_{N_K}(V)} \mu_N(\dd u) \exp \left( \beta \sqrt{N} C_N(u) 
\right)
,
\end{equation}
where
\begin{equation}
\label{eq:restricted-state-space-product}
\widetilde{S}_{N_K}(V) 
:= 
\left\{
u = u^{(1)} \shortparallel u^{(2)} \shortparallel \ldots \shortparallel
u^{(N)} \in S_{N_K}
:
\Vert u^{(k)} \Vert_2^2 \in N_{K-1} \cdot V
,
\quad
k \in [N]
\right\}
.
\end{equation}
Let us note that $\widetilde{S}_{N_K}(V) \subset S_{N_K}(V) $, and, therefore,
\begin{equation}
\label{eq:product-less-than-non-product}
Z_{N_K}(\beta, V)
\geq
\widetilde{Z}_{N_K}(\beta, V)
.
\end{equation}
The product structure \eqref{eq:restricted-state-space-product} and independence \eqref{eq:existence-comparison-field}
imply
\begin{equation}
\label{eq:decoupled-free-energy}
\begin{aligned}
\frac{1}{N_K}
\E
\left[
\log
\widetilde{Z}_{N_K}(\beta,V)[Y_{N,K}]
\right]
&
=
\frac{1}{N_K}
\E
\left[
\log
\prod_{k=1}^{N}
Z_{N_{K-1}}(\beta,V)[X^{(k)}_{N_{K-1}}]
\right]
\\
&
=
\frac{1}{N_{K-1}}
\E
\left[
\log
Z_{N_{K-1}}(\beta,V)[X_{N_{K-1}}]
\right]
.
\end{aligned}
\end{equation}
For $\eps > 0$, set $V_i := [i \eps; (i+1)\eps]$, $i \in \N$.
By the Gaussian comparison formula \cite[Proposition~2.5]{BovierKlimovskyAS2Guerra2008},
\begin{equation}
\label{eq:existence-comparison-equality}
\begin{aligned}
\frac{1}{N_K}
\E
&
\left[
\log
\widetilde{Z}_{N_K}(\beta,V_i)[X_{N_K}]
\right]
=
\frac{1}{N_K}
\E
\left[
\log
Z(\beta,V_i)[Y_{N,K}]
\right]
\\
&
+
\frac{\beta^2}{2}
\int_0^1
\dd t
\int_{
\widetilde{S}_{N_K}(V_i)
} 
\widetilde{\mathcal{G}}_{N_K}(t)(\dd u) 
\int_{
\widetilde{S}_{N_K}(V_i)
}
\widetilde{\mathcal{G}}_{N_K}(t) (\dd v)
\left[
\var X_{N_K}(u)
-
\frac{1}{N}
\sum_{k=1}^N
\var X_{N_{K-1}}(u^{(k)})
\right.
\\
&
\left.
-
\left(
\cov \left[X_{N_K}(u), X_{N_K}(v)\right]
-
\frac{1}{N}
\sum_{k=1}^N
\cov \left[X_{N_{K-1}}(u^{(k)}), X_{N_{K-1}}(v^{(k)})\right]
\right)
\right]
,
\end{aligned}
\end{equation}
where $\widetilde{\mathcal{G}}_{N_K}(t) \in \mathcal{M}_1(\widetilde{S}_{N_K})$ is the
interpolating Gibbs measure with the density
\begin{equation}
\label{eq:interpolating-gibbs-measure}
\frac{
\dd \widetilde{\mathcal{G}}_{N_K}(t)
}{
\dd \mu_{N_K}
}
(u)
=
\exp
\left(
\beta
\sqrt{N_K}
\left(
\sqrt{t} X_{N_K}(u)
+
\sqrt{1-t} Y_{N,K}(u)
\right)
\right)
,
\quad
u \in \widetilde{S}_{N_K}(V_i)
.
\end{equation}
Using \eqref{eq:long-range-covariance}, the smoothness of the correlator $D$ on $(0;+\infty)$, the fact that $D$
is non-decreasing, continuous at $0$, and $D(0)=0$, we get
\begin{equation}
\label{eq:variance-comparison}
\sup_{
u \in \widetilde{S}_{N_K}(V_i)
}
\left|
\var X_{N_K}(u)
-
\frac{1}{N}
\sum_{k=1}^N
\var X_{N_{K-1}}(u^{(k)})
\right|
\leq
D(\eps)
,
\quad
i \in \N
.
\end{equation}
As for the covariance terms, the concavity of the correlator $D$ (cf.,
Remark~\ref{rem:correlator-concavity}) and the explicit covariance
representation \eqref{eq:long-range-covariance} assure that
\begin{equation}
\label{eq:covariance-comparison}
\sup_{
u,v \in \widetilde{S}_{N_K}(V_i)
}
\left(
\cov \left[X_{N_K}(u), X_{N_K}(v)\right]
-
\frac{1}{N}
\sum_{k=1}^N
\cov \left[X_{N_{K-1}}(u^{(k)}), X_{N_{K-1}}(v^{(k)})\right]
\right)
\leq
D(\eps)
.
\end{equation}
Therefore, combining \eqref{eq:product-less-than-non-product},
\eqref{eq:decoupled-free-energy}, \eqref{eq:existence-comparison-equality},
\eqref{eq:variance-comparison} and \eqref{eq:covariance-comparison} we get
\begin{equation}
\label{eq:restricted-block-monotonicity}
\frac{1}{N_K}
\E
\left[
\log
Z_{N_K}(\beta,V_i)[X_{N_K}]
\right]
\geq
\frac{1}{N_{K-1}}
\E
\left[
\log
Z_{N_{K-1}}(\beta,V_i)[X_{N_{K-1}}]
\right]
-
C D(\eps)
,
\quad
i \in \N
.
\end{equation}
 The proof is finished by using the concentration inequality
\eqref{eq:concentratino-inequality} to remove the localisation in 
\eqref{eq:restricted-block-monotonicity}, as in
\cite[Theorem~1]{Guerra-Toninelli-Generalized-SK-2003}.

\end{proof}

\section{Comparison with cascades}
\label{sec:proof-saddle-point-product-case}

In this section, we prove Theorem~\ref{thm:free-energy}. The proof follows the
strategy that was previously implemented in
\cite[Section~5]{BovierKlimovskyAS2Guerra2008}. The appearance of the auxiliary
structures below can be made more transparent by the ``cavity'' arguments, as is
done in the seminal work of Aizenman~et~al.~\cite{AizenmanSimsStarr2006}.

\subsection{Auxiliary structures}
\label{sec:auxiliary-structures}

Consider the \emph{auxiliary index space} $\mathcal{A} = \mathcal{A}_n := \N^n$,
$n \in \N$. Let us define the \emph{projection operator} $ \mathcal{A}\ni \alpha
\mapsto [\alpha]_k := (\alpha_1, \ldots, \alpha_k) \in \N^k $, for $k \in [n]$.
It is useful to treat the elements of $\mathcal{A}$ as the \emph{leaves of the
tree} of depth $n$. We use the convention that $[\alpha]_0 = \emptyset$, where
$\emptyset$ denotes the root of the tree. Given a leaf $\alpha \in \mathcal{A}$,
we think of $\{ [\alpha_k] : k \in [n] \}$ as of the sequence of \emph{branches}
connecting the leaf $\alpha$ to the root $\emptyset$. We equip $\mathcal{A}$
with a random measure called \emph{Ruelle's probability cascade} (RPC). Let us
briefly recall the construction of the RPC, see, e.g.,
\cite{AizenmanSimsStarr2006} for more details. Note that each function $x \in
\mathcal{X}^\prime_n(r)$ can be represented as
\begin{align}
\label{eq:review:discrete-order-parameter-indicator-representation}
x(q)
=
\sum_{i=0}^n
x_{i}
\I_{
[q_i; q_{i+1})
}
(r)
,
\end{align}
where $\bar{x} = \{x_k\}_{k=0}^{n+1}$ and $\bar{q} = \{ q_k\}_{k=0}^{n+1}$
satisfy
\begin{equation}
\label{eq:review:discrete-order-parameter}
\begin{aligned}
0 =: x_0 < x_1 < \ldots < x_{n} < x_{n+1} := 1
,
\\
0 =: q_0 < q_1 < \ldots < q_n < q_{n+1} := r
.
\end{aligned}
\end{equation}
To define the RPC, we need only the sequence
$
\bar{x}
$ as in \eqref{eq:review:discrete-order-parameter}.
Consider the family of the independent (inhomogeneous) Poisson point
processes $
\{
\xi_{
k
,
[\alpha]_{k-1}
}
\mid
\alpha \in \mathcal{A}, k \in [n]
\}
$
on $\R_+$
with intensity
\begin{align}
\label{eq:review:the-ppp}
\R_+
\ni
t
\mapsto
x_k
t^{
-x_k-1
}
\in
\R_+
,
\quad
k \in [1;n] \cap \N
.
\end{align}
To each branch $[\alpha]_k$, $\alpha \in \mathcal{A}$, $k \in [n]$ of
the tree we associate the position of
the $\alpha_k$-th atom (e.g., according to the decreasing enumeration) of the
Poisson point process $\xi_{ k , [\alpha]_{k-1}
}
$. The RPC is the point process
$
\RPC
=
\RPC(x_1, \ldots, x_n)
:=
\sum_{
\alpha \in \mathcal{A}
}
\delta_{\RPC(\alpha)}
$,
where $\RPC(\alpha)$, $\alpha \in \mathcal{A}$ is obtained by multiplying
the random weights attached to the branches along the path connecting the given
leaf $\alpha \in \mathcal{A}$ with the root of the tree:
\begin{align}
\label{eq:review:rpc}
\RPC(\alpha)
:=
\prod_{k=1}^n
\xi_{
k,[\alpha]_{k-1}
}(\alpha_k)
.
\end{align}
Since
$
\sum_{\alpha \in \mathcal{A}}
\RPC(\alpha)
<
\infty
$,
the $\RPC$ can be thought of as a
finite random measure on $\mathcal{A}$ with (abusing the notation)
$\RPC(\{ \alpha \}) := \RPC(\alpha)$, for $\alpha \in \mathcal{A}$. To lighten
the notation, we keep the dependence of the RPC on $\bar{x}$ implicit.

Recall \eqref{eq:restricted-state-space}. Given the sequence $\bar{x}$ as in \eqref{eq:review:discrete-order-parameter}
and any suitable Gaussian field $ C :=
\{
C(u,\alpha)
\mid
u \in S_N
,
\alpha \in \mathcal{A}
\}
$,
let us define the \emph{extended log-partition functional} $
\Phi_N(\bar{x},V)
$
as
\begin{align}
\label{eq:remainder:comparison-functional}
\Phi_N(\bar{x},V)
[C]
:=
\frac{1}{N}
\E
\left[
\log
\left(
\int_{
S_N(V)
}
\mu(\dd u)
\int_{\mathcal{A}}
\RPC(\dd \alpha)
\exp
\left(
\beta
\sqrt{N}
C(u,\alpha)
\right)
\right)
\right]
,
\end{align}
where the $\RPC$ is induced by $\bar{x}$.

Let us use the remaining from the order parameter $x \in \mathcal{X}(r)$ bit of
information, namely, the sequence $\bar{q} = \{q_k \}_{k=0}^{n+1}$, as in
\eqref{eq:review:discrete-order-parameter}, to construct the Gaussian \emph{cavity
fields} indexed by $S_N \times \mathcal{A}$. To this end, define the
\emph{lexicographic overlap} between the configurations
$\alpha^{(1)},\alpha^{(2)} \in \mathcal{A}$ as
\begin{align}
\label{eq:introduction:lexicographic-overlap}
l(\alpha^{(1)},\alpha^{(2)})
:=
\begin{cases}
0
,
&
\alpha^{(1)}_1
\neq
\alpha^{(2)}_1
,
\\
\max
\left\{
k
\in
[N]
:
[\alpha^{(1)}]_k
=
[\alpha^{(2)}]_k
\right\}
,
&
\text{otherwise.}
\end{cases}
\end{align}
Let us define (slightly abusing the notation) the \emph{lexicographic overlap}
$
q
:
\mathcal{A}^2
\to
[0;1]
$
as
\begin{align}
\label{eq:review:limiting-grem-overlap}
q(\alpha^{(1)},\alpha^{(2)})
:=
q_{
l(\alpha^{(1)},\alpha^{(1)})
}
.
\end{align}
Given $\bar{q}$ as in
\eqref{eq:review:discrete-order-parameter}, the \emph{cavity field} is the Gaussian field $ A = A_N^{(M)} = \{A_N(u, \alpha) \mid u \in S_N, \alpha \in \mathcal{A}\} $ such that
\begin{equation}
\label{eq:cavity-field-covariance}
\cov
\left[
A^{(M)}(u,\alpha^{(1)})
,
A^{(M)}(v,\alpha^{(2)})
\right]
=
D^{\prime,M}
\left(
2 (r - q(\alpha^{(1)},\alpha^{(2)}))
\right)
\langle u,v \rangle_N
,
\quad
\alpha^{(1)}, \alpha^{(2)} \in \mathcal{A}
,
\quad
u,v \in S_N
.
\end{equation}
The existence of the cavity field $A$ is guarantied by
the following result.
\begin{lemma}[Existence of the cavity field]
\label{eq:existence-cavity-field}
For any sequence $q$ as in \eqref{eq:review:discrete-order-parameter} and large enough $M \in \R_+$,
there exists the unique (in distribution) Gaussian field satisfying
\eqref{eq:cavity-field-covariance}.
\end{lemma}
\begin{proof}
Since the distribution of the Gaussian field is completely identified by the
covariance, the uniqueness follows once we prove the existence. For this
purpose, we first construct the Gaussian field $a  = \{a^{(M)}(\alpha)\}_{\alpha
\in \mathcal{A}}$ with
\begin{align}
\label{eq:small-a-covariance}
\cov
\left[
a^{(M)}(\alpha^{(1)})
,
a^{(M)}(\alpha^{(2)})
\right]
=
D^{\prime,M}
\left(
2 (r - q(\alpha^{(1)},\alpha^{(2)}))
\right)
,
\quad
\alpha^{(1)},\alpha^{(2)} \in \mathcal{A}
.
\end{align}
To construct the field $a^{(M)}$ explicitly, we define
\begin{equation}
\label{eq:m-k-sequence}
m_k
:=
D^{\prime,M}(
2 (r - q_k)
)
,
\quad
k \in [n+1]
.
\end{equation}
The representations
\eqref{eq:short-range-incremental-covariance-representation} and
\eqref{eq:long-range-incremental-covariance-representation}, guarantee that the
sequence \eqref{eq:m-k-sequence} is non-decreasing. Therefore, we can set
\begin{equation}
\label{eq:a-field-explicit-representation}
a^{(M)}(\alpha)
:=
\sum_{k=1}^n
\left(
m_{k+1}
-
m_{k}
\right)^{1/2}
g^{(k)}_{[\alpha]_k}
,
\quad
\alpha \in \mathcal{A}
,
\end{equation}
where
$
\{
g^{(k)}_{[\alpha]_k}
\mid
\alpha \in \mathcal{A},\ k \in [n]
\}
$
are i.i.d.~standard normal random variables. A straightforward check shows
that the covariance structure of
\eqref{eq:a-field-explicit-representation} satisfies
\eqref{eq:small-a-covariance}.

To finish the construction, for $i  \in [N]$, let $a_i^{(M)} =
\{a_i^{(M)}(\alpha)\}_{\alpha \in \mathcal{A}}$ be the i.i.d.~copies of the field $a^{(M)} = \{a^{(M)}(\alpha)\}_{\alpha \in
\mathcal{A}}$. Define
\begin{equation}
\label{eq:cavity-field-explicit-representation}
A_N^{(M)}(u,\alpha)
:=
\frac{1}{\sqrt{N}}
\sum_{i=1}^N
a_i^{(M)}(\alpha)
u_i
,
\quad
u \in S_N
,
\quad
\alpha \in \mathcal{A}
.
\end{equation}
An inspection shows that the field
\eqref{eq:cavity-field-explicit-representation} satisfies
\eqref{eq:cavity-field-covariance}.
\end{proof}

\subsection{Interpolation}
\label{sec:as2:guerras-scheme}

In this section, we shall apply Guerra's comparison scheme
(cf.~\cite{Guerra2003a}) to the Gaussian field with isotropic increments
satisfying \eqref{eq:incremental-covariance}. To this end, we restrict the state
space of a particle to a thin spherical layer. This assures that the variance of
the field $X_N$ does not change much. We refer to this procedure as
\emph{localisation}. Then, we interpolate between the field of interest $X_N$ and
the cavity field \eqref{eq:cavity-field-explicit-representation} and compare the
corresponding local log-partition functions. We use the auxiliary structures
from Section~\ref{sec:auxiliary-structures}.

Given $x \in \mathcal{X}^\prime_n(r)$ and large enough $M \in \R_+$, let us consider the
following \emph{interpolating field} on the extended configuration space $S_N \times
\mathcal{A}$
\begin{align}
\label{eq:remainder:interpolating-hamiltonian}
H_t^{(M)}(u,\alpha)
:=
\sqrt{t}
X_N(u)
+
\sqrt{1-t}
A_N^{(M)}(u,\alpha)
,
\quad
t \in [0;1],
\quad
u \in S_N
,
\quad
\alpha \in \mathcal{A}
,
\end{align}
where $A_N^{(M)}$ is the cavity field with \eqref{eq:cavity-field-covariance}.
In the usual way, the field
\eqref{eq:remainder:interpolating-hamiltonian} induces the
\emph{local log-partition function}
\begin{align}
\label{eq:as2-phi-of-t-definition}
\varphi_N^{(M)}(t,x,V)
:=
\Phi_N(x,V)
\left[
H_t
\right]
,
\quad
V \subset [0;d]
,
\quad
x \in \mathcal{X}^\prime_n(r)
.
\end{align}
At the end-points of the interpolation, we obtain
\begin{align}
\varphi_N^{(M)}(0,x,V)
=
\Phi_N(\bar{x},V)
[A^{(M)}]
\quad
\text{and}
\quad
\varphi_N^{(M)}(1,x,V)
=
\Phi_N(\bar{x},V)
[X]
=:
p_N(\beta, V)
.
\end{align}
The idea is that $\Phi_N(\bar{x},V)[A^{(M)}]$ is computable due to the properties of the RPC and the hierarchical structure of the cavity field. 
Let us now disintegrate the Gibbs measure on $
V
\times
\mathcal{A}
$
induced by \eqref{eq:remainder:interpolating-hamiltonian}
into two Gibbs measures acting on $V$ and $\mathcal{A}$
separately. To this end, we define the correspondent (random) \emph{local
free energy} on $V$ as follows
\begin{align}
\label{eq:as2:psi-of-t-v-alpha}
\psi_N^{(M)}(t,x,\alpha,V)
:=
\log
\int_{
S_N(V)
}
\exp
\left[
\beta \sqrt{N}
H_t^{(M)}(u,\alpha)
\right]
\dd \mu^{\otimes N}(u)
,
\quad
\alpha \in \mathcal{A}
.
\end{align}
For $\alpha \in \mathcal{A}$, let us define the (random) \emph{local
Gibbs measure}
$
\mathcal{G}_N(t,x,\alpha,V)
\in
\mathcal{M}_1(S_N)
$
by specifying its density with respect to the a priori distribution as
\begin{align}
\frac{
\dd
\mathcal{G}_N^{(M)}(t,x,\alpha,V)
}{
\dd
\mu^{\otimes N}
}
(u)
:=
\I_{
S_N(V)
}
(u)
\exp
\left[
\beta \sqrt{N}
H_t^{(M)}(u,\alpha)
-
\psi_N^{(M)}(t,x,V,\alpha)
\right]
,
\quad
u \in S_N
.
\end{align}
Let us define the re-weighting of the RPC by means of the local free
energy
\eqref{eq:as2:psi-of-t-v-alpha}
\begin{align}
\widetilde{\mathrm{RPC}}(\alpha)
:=
\mathrm{RPC}(\alpha)
\exp
\left(
\psi_N^{(M)}(t,x,V,\alpha)
\right)
,
\quad
\alpha \in \mathcal{A}
.
\end{align}
Let us also define the \emph{normalisation operation}
$
\mathcal{N}
:
\mathcal{M}_{\text{finite}}(\mathcal{A})
\to
\mathcal{M}_1(\mathcal{A})
$
as
\begin{align}
\mathcal{N}
\left(
\eta
\right)
(\alpha)
:=
\frac{
\eta(\alpha)
}{
\sum_{
\alpha'
\in
\mathcal{A}
}
\eta(\alpha')
}
,
\quad
\alpha \in \mathcal{A}
,
\quad
\eta = (\eta_\alpha)_{\alpha \in \mathcal{A}}
\in
\mathcal{M}_{\text{finite}}(\mathcal{A})
.
\end{align}
We introduce the \emph{local Gibbs measure}
$
\mathcal{G}_N^{(M)}(t,x,V)
\in
\mathcal{M}_1
(
V \times \mathcal{A}
)
$ as follows.
We equip $V \times \mathcal{A}$ with the product topology between the Borel
topology on $V$ and the discrete topology on $\mathcal{A}$. For any measurable $
\mathcal{U}
\subset
V \times \mathcal{A}
$,
let us put
\begin{align}
\label{eq:as2:desintegrated-gibbs-measure}
\mathcal{G}_N^{(M)}(t,x,V)
\left[
\mathcal{U}
\right]
:=
\sum_{
\alpha
\in
\mathcal{A}
}
\mathcal{N}(\widetilde{\mathrm{RPC}})(\alpha)
\mathcal{G}_N^{(M)}(t,x,\alpha,V)
\{v \in V \mid (v,\alpha) \in \mathcal{U} \}
.
\end{align}
Let us define the \emph{remainder term} as
\begin{equation}
\label{eq:lcture-03:the-remainder-term-definition}
\begin{aligned}
\mathcal{R}_N^{(M)}(t,V)[x]
:=
\frac{\beta^2}{2}
\E
\Big[
&
\int
\mathcal{G}_N^{(M)}(t,x,V)(\dd u, \dd \alpha^{(1)})
\int
\mathcal{G}_N^{(M)}(t,x,V)(\dd v, \dd \alpha^{(2)})
\\
&
\left(
\frac{1}{2}
\left(
D(2(r - q(\alpha^{(1)},\alpha^{(2)})))
-
D(2(r - \langle u,v \rangle_N))
\right)
\right.
\\
&
\left.
-
D^{\prime,M}(2(r -  q(\alpha^{(1)},\alpha^{(2)})))
(
q(\alpha^{(1)},\alpha^{(2)}))
-
\langle u,v \rangle_N
)
\right)
\Big]
.
\end{aligned}
\end{equation}
Given $r \in (0;d]$, let us denote 
\begin{equation}
V_\epsilon := (r-\epsilon; r+\epsilon)
.
\end{equation}
Define the \emph{local remainder term}
as
\begin{align}
\label{eq:remainder:local-limiting-guerra-remainder}
\mathcal{R}^{(M)}(r)[x]
:=
\lim_{
\eps
\downarrow
+0
}
\lim_{
N
\uparrow
+\infty
}
\int_{0}^{1}
\mathcal{R}_N^{(M)}(t,V_\eps)
\dd t
,
\quad
x \in \mathcal{X}^\prime_n(r)
.
\end{align}
The main step in the proof of Theorem~\ref{thm:free-energy} is the following.
\begin{lemma}[Comparison with cascades]
\label{thm:as2:guerras-interpolation}
Given $r \in (0;d]$,
for any $x \in \mathcal{X}^\prime_n(r)$,
as $\eps \downarrow +0$, and $M \uparrow +\infty$,
\begin{align}
\label{eq:as2:guerras-interpolation-speed}
\frac{\partial}{\partial t}
\varphi_N^{(M)}(t,x,V_\eps)
=
&
-
\mathcal{R}^{(M)}(r)[x]
-
\frac{\beta^2}{2}
\sum_{
k=1
}^{
n
}
x_k
\left(
\theta_r^{(M)}(q_{k+1})
-
\theta_r^{(M)}(q_k)
\right)
+
\mathcal{O}(\eps)
+
\mathcal{O}(1/M)
,
\end{align}
where 
\begin{equation}
\label{eq:remainder-positivity}
\mathcal{R}^{(M)}(r)[x]
\geq 0
.
\end{equation}
\end{lemma}
\begin{proof}
Fix some $r \in (0;d]$. Using the notation \eqref{eq:effective-sk-correlator} and smoothness of $D$ on $(0;+\infty)$, we have
\begin{equation}
\label{eq:variances-concrete}
\var X(u)
=
G_r(r)+\mathcal{O}(\eps)
,
\quad
\var A(u,\alpha)
=
r G_r^\prime(r)+\mathcal{O}(\eps)
,
\quad
u \in V_\epsilon
,
\quad
\alpha \in \mathcal{A}
.
\end{equation}
and
\begin{equation}
\label{eq:isotropic-comparison-covariance-structures}
\begin{aligned}
\cov \left[
X(u), X(v)
\right]
&
=
G_r(\langle u,v\rangle_N)
,
\\
\cov \left[
A(u,\alpha^{(1)})
,
A(v,\alpha^{(2)})
\right]
&
=
G_r^\prime(q(\alpha^{(1)},\alpha^{(2)}))
\langle u,v \rangle_N
.
\end{aligned}
\end{equation}
Applying the abstract Gaussian interpolation formula (see, e.g., \cite[Proposition~2.5]{BovierKlimovskyAS2Guerra2008}) 
to the field $X_N$ and the cavity field \eqref{eq:cavity-field-explicit-representation}, we obtain
\begin{equation}
\label{eq:interpolation-remainer-semi-abstract}
\begin{aligned}
\frac{\partial}{\partial t}
\varphi_N(t,x,V_\eps(r))
=
\frac{\beta^2}{2}
\E\left[
\int
\mathcal{G}_N(t,x,V)(\dd u, \dd \alpha^{(1)})
\int
\mathcal{G}_N(t,x,V)(\dd v, \dd \alpha^{(2)})
\right.
\\
\left.
\left(
\var X(u) - \var A(u,\alpha)
-
\cov \left[ X(u),X(v) \right] + \cov \left[ A(u,\alpha^{(1)}),A(v,\alpha^{(2)}) \right]
\right)
\right]
+
\mathcal{O}(\eps)
.
\end{aligned}
\end{equation}
Using \eqref{eq:variances-concrete} and \eqref{eq:isotropic-comparison-covariance-structures}, we get
\begin{equation}
\label{eq:remainder-full-integrand}
\begin{aligned}
&
\var X(u) - \var A(u,\alpha)
-
\cov \left[ X(u),X(v) \right] + \cov \left[ A(u,\alpha^{(1)}),A(v,\alpha^{(2)}) \right]
\\
&
=
G_r(r)-r G_r^\prime(r)
-
\left(
G_r(q(\alpha^{(1)},\alpha^{(2)})) 
-
q(\alpha^{(1)},\alpha^{(2)})
G_r^\prime(q(\alpha^{(1)},\alpha^{(2)}))
\right)
\\
&
\quad
-
\left[
G_r(\langle u, v\rangle_N)
-
G_r(q(\alpha^{(1)},\alpha^{(2)}))
-
G_r^\prime(q(\alpha^{(1)},\alpha^{(2)}))
\left(
\langle u, v\rangle_N
-
q(\alpha^{(1)},\alpha^{(2)})
\right)
\right]
.
\end{aligned}
\end{equation}
Comparing \eqref{eq:effective-sk-correlator} and \eqref{eq:theta-def}, we note
\begin{equation}
G_r(q)- s G_r^\prime(q)
=
D(r)
+
\theta_r(q)
,
\quad
q \in \R_+
.
\end{equation}
We have (cf. the proof of \cite[Lemma~5.2]{BovierKlimovskyAS2Guerra2008})
\begin{equation}
\label{eq:linear-part-calculation}
\begin{aligned}
&
\E\left[
\int
\mathcal{G}_N^{(M)}(t,x,V_\eps)(\dd u, \dd \alpha^{(1)})
\int
\mathcal{G}_N^{(M)}(t,x,V_\eps)(\dd v, \dd \alpha^{(2)})
(
\theta_r(r)
-
\theta_r(q(\alpha^{(1)},\alpha^{(2)}))
)
\right]
\\
&=
\E\left[
\int
\mathcal{N}(\widetilde{\mathrm{RPC}})(\dd \alpha^{(1)})
\int
\mathcal{N}(\widetilde{\mathrm{RPC}})(\dd \alpha^{(2)})
(
\theta_r^{(M)}(r)
-
\theta_r^{(M)}(q(\alpha^{(1)},\alpha^{(2)}))
)
\right]
\\
&
=
\sum_{k=1}^n
x_k
(\theta_r^{(M)}(q_{k+1})-\theta_r^{(M)}(q_{k}))
.
\end{aligned}
\end{equation}
By \eqref{eq:effective-sk-correlator},
\begin{equation}
\label{eq:remainder-integrand}
\begin{aligned}
&
G_r(\langle u, v\rangle_N)
-
G_r(q(\alpha^{(1)},\alpha^{(2)}))
-
G_r^\prime(q(\alpha^{(1)},\alpha^{(2)}))
\left(
\langle u, v\rangle_N
-
q(\alpha^{(1)},\alpha^{(2)})
\right)
\\
&
=
\frac{1}{2}
\left(
D(2(r - q(\alpha^{(1)},\alpha^{(2)})))
-
D(2(r - \langle u,v \rangle_N))
\right)
\\
&
\quad
-
D^\prime(2(r -  q(\alpha^{(1)},\alpha^{(2)})))
(
q(\alpha^{(1)},\alpha^{(2)}))
-
\langle u,v \rangle_N
)
.
\end{aligned}
\end{equation}
Combining \eqref{eq:linear-part-calculation}, \eqref{eq:remainder-full-integrand}, \eqref{eq:remainder-integrand} and \eqref{eq:interpolation-remainer-semi-abstract}, we get \eqref{eq:as2:guerras-interpolation-speed}. 
Due to Remark~\ref{rem:correlator-concavity}, the function $G$ is convex. Therefore,
\begin{equation}
\label{eq:remainder-convexity}
G_r(\langle u, v\rangle_N)
-
G_r(q(\alpha^{(1)},\alpha^{(2)}))
-
G_r^\prime(q(\alpha^{(1)},\alpha^{(2)}))
\left(
\langle u, v\rangle_N
-
q(\alpha^{(1)},\alpha^{(2)})
\right)
\geq 0
.
\end{equation}
Inequality~\eqref{eq:remainder-positivity} follows from
\eqref{eq:remainder-convexity}.
\end{proof}

\subsection{Regularisation and localisation}
\label{sec:regularisation-and-localisation}

In this section, we finish the proof of Theorem~\ref{thm:free-energy}.

\begin{lemma}[Regularisation, well-definiteness]
\label{lem:regularisation-works}
For any $x \in \mathcal{X}^\prime_n(r)$, 
\begin{equation}
\label{eq:regularisation-works}
\lim_{M \uparrow +\infty}
\left[
\lim_{x_n \uparrow 1-0}
\left(
\lim_{\epsilon \downarrow +0}
\Phi_N(\bar{x},V_\epsilon)[\tilde{A}]
-
\frac{\beta^2}{2}
\sum_{
k=1
}^{
n
}
x_k
\left(
\theta_r^{(M)}(q_{k+1})
-
\theta_r^{(M)}(q_k)
\right)
\right)
\right]
<
\infty
.
\end{equation}
\end{lemma}
\begin{proof}
Recall \eqref{eq:a-field-explicit-representation}. Given $x \in
\mathcal{X}^\prime_n(r)$, large enough given $M > 0$, as $\epsilon \downarrow
+0$ and $x_n \uparrow 1-0$, we have
\begin{equation}
\label{eq:non-linear-term-large-m-asymptotics}
\varphi_N^{(M)}(0,x,V_\eps)
= 
\frac{\beta^2}{2}
\left(M-D^\prime(2(r-q_n))\right)r
+
\Phi_N(\bar{x},V_\eps)[\tilde{A}]
+
\mathcal{O}(\epsilon)
+
\mathcal{O}(1-x_n)
,
\end{equation}
where
$
\tilde{A}(u,\alpha) := \frac{1}{\sqrt{N}}
\sum_{i=1}^N
\tilde{a}_i^{(M)}(\alpha)
u_i
$, and $\{ \tilde{a}_i\} $  are i.i.d. copies of 
\begin{equation}
\tilde{a}^{(M)}(\alpha)
:=
\sum_{k=1}^{n-1}
\left(
m_{k+1}
-
m_{k}
\right)^{1/2}
g^{(k)}_{[\alpha]_k}
,
\quad
\alpha \in \mathcal{A}
.
\end{equation}
Using the definition~\eqref{eq:theta-def},
for large enough given $M > 0$, as $x_n \uparrow 1-0$, we get
\begin{equation}
\label{eq:linear-term-large-m-asymptotics}
x_{n}
\left(
\theta_r^{(M)}(q_{k+1})
-
\theta_r^{(M)}(q_k)
\right)
=
\left(M-D^\prime(2(r-q_n))\right)r
-
\frac{1}{2}D(2(r-q_n))
+
\mathcal{O}(1-x_n)
.
\end{equation}
Combining~\eqref{eq:non-linear-term-large-m-asymptotics} and
\eqref{eq:linear-term-large-m-asymptotics}, we note that the unbounded in $M$
terms in \eqref{eq:regularisation-works} cancel out and therefore \eqref{eq:regularisation-works} holds.
\end{proof}

\begin{lemma}[Localisation, large deviations and cascades]
\label{lem:ldp}
For any $x \in \mathcal{X}^\prime_n(r)$, 
\begin{equation}
\lim_{\epsilon \downarrow +0}
\varphi_N^{(M)}(0,x,V_\eps)
=
\inf_{\lambda \in \R}
\left[
f_{r,x,g_\lambda}^{(M)}(0,0)
-\lambda r
\right]
.
\end{equation}
\end{lemma}
\begin{proof}
This is a standard computation (cf., e.g.,~\cite[Lemma
6.2]{AizenmanSimsStarr2006}), using the well-known averaging properties of the
RPC (see, e.g.,~\cite[(5.27)]{BovierKlimovskyAS2Guerra2008}) and the quenched
large deviations principle as is done in
\cite[Sections~3-5]{BovierKlimovskyAS2Guerra2008}.
\end{proof}

\begin{proof}[Proof of Theorem~\ref{thm:free-energy}]
Combining Lemmata~\ref{thm:as2:guerras-interpolation}, \ref{lem:ldp} and
\ref{lem:regularisation-works}, we obtain Theorem~\ref{thm:free-energy}.
\end{proof}

\section{Outlook}
\label{sec:outlook}
Combining the methods of Talagrand~\cite{TalagrandParisiFormula2006} with
Theorem~\ref{thm:free-energy}, we can show that the remainder term in
\eqref{eq:free-energy-variational-plus-remainder-term} vanishes at the
saddle-point. This implies that, in fact, the equality holds in
\eqref{eq:free-energy-upper-bound}. Summarising, we arrive
at the following result.
\begin{theorem}[Parisi-type formula]
\label{thm:parisi-type-formula}
In the case of the product state space \eqref{eq:product-space}, for all $\beta
\in \R$,
\begin{equation}
\label{eq:parisi-type-formula}
p(\beta)
=
\sup_{r \in [0;d]} \inf_{x \in \mathcal{X}(r)}
\mathcal{P}(\beta, r, g)[x]
,
\quad
\text{almost surely.}
\end{equation}
\end{theorem}

Parallel to the product state space \eqref{eq:product-space}, one can consider the 
\emph{rotationally invariant state space}:
\begin{equation}
\label{eq:euclidean-ball}
S_N := \{ u \in \R^N : \Vert u \Vert_2 \leq L
\sqrt{N} \},
\quad
L > 0.
\end{equation}
In this case, we assume that
the a priori measure $\mu_N \in \mathcal{M}_{\text{finite}}(S_N)$ has the
density
\begin{align}
\label{eq:ball-external-field}
\frac{
\dd \mu
}{
\dd \lambda
}
(u)
:=
\exp
\left(
\sum_{i=1}^N f(u_i)
\right)
,
\quad
u = (u_i)_{i=1}^N \in \R^N
,
\quad
f: \R \to \R
\end{align}
 with respect to the Lebesgue measure $\lambda$ on $\R^N$. Let the
function $f$ be of the form $f(u) := h_1 u - h_2 u^2$, where $h_1
\in \R$ and $h_2 \in \R_+$ are given constants. Let us note that in case
\eqref{eq:euclidean-ball}, $d = L^2$.

In the case of the rotationally invariant state space
\eqref{eq:euclidean-ball}, one can obtain a more explicit representation for the
Parisi functional \eqref{eq:local-parisi-functional}, which does not require any
regularisation. Given $x \in \mathcal{X}(r)$, define
 $ q_{\mathrm{max}} := q_{\mathrm{max}}(x) := \sup \big\{ q \in [0;r] : x(q) < 1
\big\} $. Consider the Crisanti-Sommers type functional (cf.
\cite[(A2.4)]{CrisantiSommers1992} and \cite[(47)]{FyodorovSommers2007})
\begin{equation}
\label{eq:crisanti-sommers-functional}
\begin{aligned}
\mathcal{CS}(\beta, r)[x]
:=
&
\frac{1}{2}
\left[
\log(r - q_{\mathrm{max}})
+
\int_0^{q_{\mathrm{max}}}
\frac{\dd q}{
\int_q^{r}
x(s)
\dd s
}
+
h_1^2 \int_0^r x(q) \dd q
-
h_2 r
\right]
\\
&
+
\frac{\beta^2}{2}
\left(
D^\prime(2(r-q_\mathrm{max}))
+
\int_{
0
}^{
q_\mathrm{max}
}
D^\prime(2(r-q))x(q)\dd q
\right)
,
\quad
x \in \mathcal{X}(r)
.
\end{aligned}
\end{equation}
By reducing the case of the rotationally invariant state space to the product
state space case using a large deviations argument (an idea exploited in \cite{TalagrandSphericalSK}), one arrives at the
following result.
\begin{theorem}[Fyodorov-Sommers formula]
\label{thm:fyodorov-sommers-formula}
In the case of the rotationally invariant state space \eqref{eq:euclidean-ball},
for all $\beta \in \R_+$, $h_1 \in \R$, $h_2 \in \R_+$, there exists unique
$r^* \in [0;d]$ and unique $x^* \in \mathcal{X}(r)$ such that
\begin{equation}
\label{eq:fyodorov-sommers-formula}
p(\beta)
=
\max_{r \in [0;d]} \min_{x \in \mathcal{X}(r)}
\mathcal{CS}(\beta, r)[x]
=
\mathcal{CS}(\beta, r^*)[x^*]
,
\quad
\text{almost surely.}
\end{equation}
\end{theorem}
The proofs of Theorems~\ref{thm:parisi-type-formula}
and~\ref{thm:fyodorov-sommers-formula} are beyond the scope of this short
communication and will be reported on elsewhere.
\begin{remark}
The Crisanti-Sommers type functional~\eqref{eq:crisanti-sommers-functional}
corresponds to the a priori distribution~\eqref{eq:ball-external-field}, which
represents the linear combination of linear and quadratic external fields.
Formula~\cite[(47)]{FyodorovSommers2007} was derived under the assumption of the quadratic external
field, whereas formula~\cite[(A2.4)]{CrisantiSommers1992} was obtained for the spherical
SK model with the linear external field.
\end{remark}
\begin{remark}
The explicit form of the functional \eqref{eq:crisanti-sommers-functional}
assures that it is strictly convex with respect to $x \in \mathcal{X}(r)$. In
contrast, convexity of the functional \eqref{eq:local-parisi-functional} is (to
the author's best knowledge) open, see~\cite{PanchenkoParisiMeasures2004} and
\cite[Theorem~6.4]{BovierKlimovskyAS2Guerra2008} for partial results.
\end{remark}

\appendix

\section{Characterisation of the correlators}
We recall some facts about high-dimensional Gaussian processes with isotropic
increments. The following result can be found in the work \cite{yaglom_classes_1957} of A.M.~Yaglom (see also \cite{Yaglom1987}).
\begin{theorem}
\label{thm:characterisation-correlation-functions}
If $X$ is a Gaussian random field with isotropic increments that satisfies
\eqref{eq:incremental-covariance}, then one of the following two cases holds:
\begin{enumerate}

\item
\label{i:isotropic-field}
\emph{Isotropic field.} There exists the \emph{correlation function} $B: \R_+
\to \R$ such that
\begin{align}
\label{eq:stationary-correlation-structure}
\E
\left[
X_N(u)
X_N(v)
\right]
=
B\left(\frac{1}{N}\Vert u - v \Vert_2^2\right)
,
\quad
u,v \in \Sigma_N
,
\end{align}
where the function $B$ has the representation
\begin{align}
\label{eq:correlation-function-representation}
B(r)
=
c_0
+
\int_0^{+\infty}
\exp
\left(
-t^2 r
\right)
\nu(\dd t)
,
\end{align}
where $c_0 \in \R_+$ is a constant and $\nu \in
\mathcal{M}_{\mathrm{finite}}(\R_+)$ is a non-negative finite measure.
In this case, the function $D$ in \eqref{eq:incremental-covariance} is expressed
in terms of the correlation function $B$ as
\begin{equation}
\label{eq:short-range-incremental-covariance-representation}
D(r)= 2 (B(0)-B(r))
.
\end{equation}

\item
\label{i:isotropic-increments}
\emph{Non-isotropic field with isotropic increments}. The function $D$ in
\eqref{eq:incremental-covariance} has the following representation
\begin{align}
\label{eq:long-range-incremental-covariance-representation}
D(r)
=
\int_0^{+\infty} \big[1-\exp\big(-t^2 r \big)\big] \nu(\dd t) + A \cdot r
,
\quad
r \in \R_+
,
\end{align}
where $A \in \R_+$ is a constant and $\nu \in \mathcal{M}((0;+\infty))$ is
a $\sigma$-finite measure with
\begin{align}
\label{eq:long-range-integrability-condition}
\int_0^{+\infty} \frac{t^2 \nu(\dd t)}{t^2+1}  < \infty
.
\end{align}
\end{enumerate}
\end{theorem}
\begin{remark}
\label{rem:short-long-range-cases} In
Theorem~\ref{thm:characterisation-correlation-functions}, assuming $c_0 = 0$,
case~\ref{i:isotropic-field} is sometimes referred to as the \emph{short-range}
one which reflects the decay of correlations: $B(r)\downarrow +0$, as $r
\uparrow +\infty$. This fact follows from the
representation~\eqref{eq:correlation-function-representation}. Correspondingly,
case~\ref{i:isotropic-increments} is called the \emph{long-range} one, since
here, assuming $X(0)=0$, the correlation structure is
\begin{align}
\label{eq:long-range-covariance}
\E
\left[
X_N(u) X_N(v)
\right]
=
\frac{1}{2}
\left(
D_N(\Vert u \Vert_2^2)
+
D_N(\Vert v \Vert_2^2)
-
D_N(\Vert u - v \Vert_2^2)
\right)
,
\quad
u,v \in \R^N
.
\end{align}
Equation \eqref{eq:long-range-covariance} in combination with the
representation~\eqref{eq:long-range-incremental-covariance-representation}
implies that the correlations of the field $X_N$ do not decay, as
$\Vert u - v \Vert \to +\infty$. 
\end{remark}
\begin{remark}
\label{rem:correlator-concavity}
Theorem~\ref{thm:characterisation-correlation-functions} implies that the
function $D$ appearing in \eqref{eq:incremental-covariance} is necessarily
concave, infinitely differentiable, and non-decreasing on $(0;+\infty)$.
\end{remark}

\noindent\textbf{Acknowledgements.} The author is grateful to Prof.~Yan~V.~Fyodorov for
useful remarks and his interest in this work. Kind hospitality of the Hausdorff
Research Institute for Mathematics, where a part of the present work was done, is gratefully
acknowledged.

\bibliographystyle{plain}
\bibliography{bibliography}

\end{document}